\documentclass[11pt,english]{smfart}
\usepackage[T1]{fontenc}
\usepackage[french,english]{babel}
\usepackage{layout}
\usepackage{amstext,amssymb,amscd}
\usepackage{smfthm}
\usepackage[all]{xy}
\newcommand{\incl}[1][r]{\ar@<-0.2pc>@{^(-}[#1] \ar@<+0.2pc>@{-}[#1]}

\usepackage{youngtab}
\usepackage{graphicx,epsfig}
\usepackage{tabularx}

\def\R{{\mathbb R}}
\def\C{{\mathbb C}}
\def\Z{{\mathbb Z}}

\def\S{{\mathbb S}}
\def\A{{\mathbb A}}

\def\g{{\mathfrak g}}
\def\p{{\mathfrak p}}
\def\l{{\mathfrak l}}
\def\k{{\mathfrak k}}

\def\u{{\mathfrak u}}
\def\q{{\mathfrak q}}
\def\t{{\mathfrak t}}
\def\vhi{{\varphi}}
\def\bk{{\backslash}}

\def\SL{\text{SL}(2,\C)}

\author{Mathieu Cossutta}
\address{ D\'epartement de Math\'ematiques et Applications (DMA)\\ CNRS : UMR8553\\Ecole Normale Sup\'erieure de Paris}
\email{mathieu.cossutta@ens.fr}
\title{Theta lifting for some cohomologicaly induced representations}

\begin{document}
\renewcommand{\labelitemi}{$\bullet$}
\frontmatter
\begin{abstract}
In his paper \cite{Li1}, Jian-Shu Li proved that a certain kind of
cohomological representations of $U(a,b)$ is automorphic. In this paper,
this result is generalized to a more general class of cohomological representations of
this group. It comes from the fact that these cohomological
representations are the image by the theta correspondance of some
representations obtained by cohomological induction. In proving this theorem,
we also prove some cases of Adams conjecture \cite{Adams}.
\end{abstract}

\maketitle

\section{Introduction}\label{sect0}
Let $a,b\in\Z_{\geq 0}$ and $n=a+b$. Let $G=U(a,b)$ viewed as the group of matrices:
$$\left\{M\in M(n,\C)|\ \overline{^tM}I_{a,b}M=I_{a,b}\right\}$$
where:
$$I_{a,b}=
\left(
\begin{array}{cc}
I_a&0\\
0&-I_b
\end{array}
\right).
$$
We have a diagonal embedding of $K=U(a)\times U(b)$ in $G$. $K$ is a
maximal subgroup of $G$ and the quotient
$$X_{a,b}=G/K,$$
is the symmetric space associated to $G$. In the paper \cite{Li1}, Jian-Shu Li proves the
most general known results about the construction of non trivial cohomology
classes for manifolds
$$S(\Gamma)=\Gamma\bk X_{a,b}$$
where $\Gamma$ is a discrete cocompact subgroup of $G$ acting freely on $X_{a,b}$. For example, he
proves the following theorem.
\begin{theo}\cite{Li1}\label{theo0.1}
Let $m,s,k,l,t\in\Z_{\geq 0}$ such that
$$k+l\leq a,\ s+m\leq b, m+s+k+l< 2(a+b)\text{ and }t\leq mk+sl$$
then there exists a discrete cocompact subgroup $\Gamma$ of $\text{U}(a,b)$
such that
$$H^{R^+,R^-}(S(\Gamma),\C)\neq 0,$$
where $R^+=k(q-s)+ps-t$ and $R^-=t+(q-s)l+(p-k-l)m$.
\end{theo}
We give a generalization of this theorem (see theorem \ref{theo4.5}). Before
stating our result, we recall the proof of this theorem. Let $(\q=\u\oplus \l,L)$ be a parabolic theta
stable algebra (see definition \ref{defi2.1}) and $\lambda$ be the differential of a unitary character
of $L$. Let 
$$R(\q)=\dim\u\cap\p.$$
According to Knapp and Vogan \cite{KV}, one can then define by cohomological
induction a representation $A_\q(\lambda)$. Under the
condition that
\begin{eqnarray}\label{eq0.1}
(\lambda,\tau)\geq 0\ \forall \tau\in\u,
\end{eqnarray}
 the representation $A_\q(\lambda)$ is non zero unitary and
 irreducible. Let  $\Gamma$ be a cocompact discrete subgroup of $G$
acting freely on $X$, then by Matsushima's formula
\begin{eqnarray}\label{eq0.2}
H^*(\Gamma\bk X)=\oplus_{\q}
m(A_\q(0),\Gamma)\hom_K(\bigwedge^*\mathfrak{p},A_\q(0)).
\end{eqnarray}
A more precise version of theorem \ref{theo0.1} is:
\begin{theo}\cite{Li1}\label{theo0.2}
Let $(\q,L)$ be a parabolic theta stable algebra of $U(a,b)$ such that
$$L=U(x,y)\times \text{a compact group},$$
and suppose that $2(x+y)>a+b$ then there exists a discrete cocompact
subgroup of $U(a,b)$ such that:
$$m(A_\q,\Gamma)\neq 0.$$
\end{theo}
The proof is based on the use of local and
global theta correspondance. In fact Jian-Shu Li proved the following 
archimedean result on theta correspondance.

\begin{theo}\cite{Li3}\label{theo0.3} Let $(\q,L)$ be a parabolic theta stable algebra of $U(a,b)$
  such that 
$$L=U(x,y)\times \text{a compact group},$$
and $\lambda$ be the differential of a unitary character of $L$ verifying
(\ref{eq0.1}). Suppose that $(x+y)\geq a+b$ then there exists $a',b'$ with
$a'+b'=a+b-(x+y)$ and a discrete series $\pi_d$ of $U(a',b')$ such that:
$$\theta(\pi_d)=A_\q(\lambda).$$
\end{theo}
It is now classical to deduce theorem \ref{theo0.2} from theorem \ref{theo0.3} by using global theta
correspondance and  De George and Wallach result \cite{DW1} on
multiplicities of discrete series. So to
generalize theorem \ref{theo0.2}, we have to prove that more general
representations of type $A_\q(\lambda)$ are in the image of the theta
correspondance. In that purpose, one has to find a candidate for the equation
\begin{eqnarray}\label{eq0.3}
\theta(\pi)=A_\q(\lambda).
\end{eqnarray}
To solve this equation are introduced Arthur parameter and Arthur
packet for cohomologically induced representation and the Adams
conjectures are stated. This is done in sections \ref{sect1},\ref{sect2} and \ref{sect3}. This allows to describe in a
natural way a candidate $\pi$ for equation (\ref{eq0.3}). Our main theorem is then
the following:

\begin{theo}\label{theo0.4}
Let $(\q,L)$ be a standard parabolic theta stable algebra  with 
$$L=\prod_{i=1}^rU(a_i,b_i)$$
and $\lambda$ be the differential of a unitary character of $L$ verifying
(\ref{eq0.1}). Let $r_0\in\{1,\dots,r\}$ such that
$$a_{r_0}+b_{r_0}=\max_i(a_i+b_i).$$
Let:
$$
\left\{
\begin{array}{ll}
a'=\sum_{i<r_0}a_i+\sum_{i>r_0}b_i\text{ and }\\
b'=\sum_{i>r_0}a_i+\sum_{i<r_0}b_i
\end{array}
\right..
$$
We suppose that 
$$
\sum_{i\neq r_0}a_i+b_i\leq \min(a,b),
$$
then $A_\q(\lambda)$ is a theta lift from $U(a',b')$ of a representation 
$A_{\q'}(\lambda')$ cohomologically induced from a parabolic theta stable algebra $(\q',L')$ verifying:
$$L'=\prod_{i=1}^{r_0-1}U(a_i,b_i)\times\prod_{r_0+1}^{r}U(b_i,a_i).$$
\end{theo}

Since a representation $A_{\q}(\lambda)$ is a discrete serie if and only
if $L$ is compact, this theorem implies that some cohomological
representations are obtained
from a discrete series by a certain number of theta lifts. It can then be globalized to
prove that some new $A_\q(\lambda)$ verify $m(A_\q,\Gamma)\neq 0$.  We
introduce the notion of a convergent parabolic theta stable algebra
(definition \ref{defi4.2}).
\begin{theo} Let $\q$ be a convergent parabolic theta stable algebra of
  $G$. There exists a discrete cocompact subgroup $\Gamma$ acting freely on
  $X_{a,b}$ such that:
$$m(A_{\q}(0),\Gamma)\neq 0.$$
\end{theo}
The article goes as follows. Some facts about Arthur parameter are recalled
in the first part and some about $A_\q(\lambda)$ in the second one. The
proof of our main archimedean results (theorem
\ref{theo3.3}) is developped in the third part,
using the facts recalled in the two previous one. In the last part, some
global applications are given.

\section{Arthur packets}\label{sect1}
Arthur parameters are representations of Weil groups in
$L$-groups. Some useful definitions are recalled in the case of unitary
groups.  
\begin{defi}\label{defi1.1}
The Weil group of $\C$ (resp. $\R$) is defined by:
\begin{itemize}
\item $W_\C=\C^*$,
\item $W_\R=\C^*\cup j\C^*,$ with $j^2=-1$ and $jcj^{-1}=\overline{c}$.
\end{itemize}
\end{defi}
There is a canonical inclusion of $W_\C$ in $ W_\R$ sitting in an exact sequence:
\begin{equation}\label{seq}
1\rightarrow W_\C\rightarrow W_\R\rightarrow \{\pm 1\}\rightarrow 1.
\end{equation}

\begin{defi}\label{defi1.2}

The $L$-group of $U(a,b)$ is the semi-direct product:
$$\text{GL}(n,\C)\rtimes W_\R,$$
with $z\in\C^\times$ acting trivially on $\text{GL}(n,\C)$ and $j$ acting by:
$$x\rightarrow \Phi_n\ ^tx^{-1}\Phi_n^{-1},$$
where $\Phi_n$ is the matrix defined by
$$(\Phi_n)_{i,j}=\delta_{j,n+1-j}(-1)^{n-i}.$$
\end{defi}
\begin{defi}\label{defi1.3}
An Arthur parameter is a homomorphism:
$$\psi: W_\R\times\text{SL}(2,\C)\rightarrow\ ^LG,$$
satisfying:
\begin{enumerate}
\item the diagram
$$\xymatrix{ W_\R\times\text{SL}(2,\C) \ar[rr]^{\psi} \ar[rd] &&\ ^LG \ar[ld] \\ & W_\R}$$
is commutative,
\item the restriction of $\psi$ to ${\text{SL}(2,\C)}$ is algebraic and
\item the image of the restriction of $\psi$ to ${W_\R}$ is quasi-compact.
\end{enumerate}
\end{defi}
Let $\psi$ be an Arthur parameter. In this paper, we are interested by the restriction of $\psi$ to:
$$W_\C\times \text{SL}(2,\C).$$
Let $\mu$ be the character of $\C^*$ given by:
$$\mu(z)=\frac{z}{\overline{z}}$$
and $\sigma_k$ be the unique representation of dimension $k$ of
$\text{SL}(2,\C)$. Cases where
$$\psi_{|W_\C\times \text{SL}(2,\C)}=\oplus_{j}\mu^{k_j}\otimes\sigma_{n_j}$$
with $k_j$ half integral, will be studied. 

One hopes to associate to every Arthur parameter,
a packet $\Pi(\psi)$ of irreducible unitarisable $(\g,K)$-modules with
properties imposed by the trace formula (see \cite{Arthur}). These sets should not form a partition of the unitary dual of $G$,
but their elements should appear in space $L^2(\Gamma\bk G)$ for some
discrete arithmetic $\Gamma$.  The most simple example is the following.
\begin{exem}\label{exem1.1}$\ $
If:
$$\psi_{|\C^\times\times\text{SL}(2,\C)}=\mu(z)^{k}\otimes\sigma_{n},$$
then:
\begin{eqnarray}
\Pi(\psi)=\left\{\left(\det\bullet\right)^k\right\}.
\end{eqnarray}
\end{exem}

Bergeron and Clozel have conjectured in \cite{BC} that the elements of
$\Pi(\psi)$ have the same infinitesimal character and that we can determine
it simply from $\psi$. We state the precise conjecture.   
\begin{rema}
The subscript $0$ is used for real Lie
algebras and no subscript for complex one. The Lie algebra $\mathfrak{g}_0$ of
$G$ is:
$$\left\{M\in M(n,\C)|\ \overline{^tM}I_{a,b}+I_{a,b}M=0\right\}.$$

The Lie algebra $\mathfrak{k}_0$ of $K$ is:
$$\mathfrak{k}_0=\left\{\left(
\begin{array}{cc}
A&0\\
0&B
\end{array}
\right)|\ \overline{^tA}+A=\overline{^tB}+B=0\right\}.$$
A compact Cartan algebra of both $\k_0$ and $\g_0$ is:
$$\mathfrak{t}_0=\left(\left(\begin{array}{cccccc
}
x_1&\\
&\ddots&&&\\
&&x_a&&&\\
&&&y_1&&\\
&&&&\ddots&\\
&&&&&y_b
\end{array}\right)|\ x_i,y_j\in\imath\R\right).$$
\end{rema}
So by the Harish-Chandra isomorphism, the infinitesimal character
of an irreducible $(\g,K)$-module can be seen 
as an element of $\C^{a+b}$ (up to permutation).

\begin{defi}\label{defi1.4} Let $\psi$ be an Arthur parameter. If:
$$\psi_{|W_\C\times \text{SL}(2,\C)}=\oplus_{j}\mu^{k_j}\otimes\sigma_{n_j},$$
the element
$$\left(k_j+\frac{n_j+1-2t}{2}\right)\ t=1,\dots,n_j$$
of $\C^{a+b}$ is called the infinitesimal character of $\psi$.
\end{defi}

\begin{conj}[Bergeron and Clozel \cite{BC}]
Let $\psi$ be an Arthur parameter then every elements of $\Pi(\psi)$ has the infinitesimal character of $\psi$.
\end{conj}

\section{Representation $A_\q(\lambda)$}\label{sect2}
Let $\p_0$ be the orthogonal complement for the Killing form of $\k_0$ in $\g_0$. So
$$\p=\left\{\left(
\begin{array}{cc}
0&A\\
B&0
\end{array}
\right)|\ A,\ ^tB\in M_{a,b}(\C)\right\}.$$

Let $\Delta(\mathfrak{g},\mathfrak{t})$ be the set of roots of $\mathfrak{t}$ in
$\mathfrak{g}$. One has that:

\begin{equation}\label{eq2.1}
\begin{split}
\Delta(\g,\t)&=\Delta(\k,\t)\cup\Delta(\p,\t),\\
\Delta(\mathfrak{k},\mathfrak{t})&=\left\{x_i-x_j|\ i\neq j\}\cup \{y_i-y_j|\
i\neq j\right\}\text{ and}\\
\Delta(\mathfrak{p},\mathfrak{t})&=\left\{x_i-y_j|i,j\right\}.
\end{split}
\end{equation}
We let $\mathfrak{g}^\tau$ be the eigenspace associated to a root
$\tau$. 
\begin{rema}\label{rema2.1}

Since $K$ is compact, for all $H\in\imath\t_0$ and $\tau\in \Delta(\g,\t)$,
one has:
$$\tau(H)\in\R.$$
\end{rema}
\subsection{Parabolic theta stable algebra}\label{ssect2.1} Let $H\in\imath\t_0$. One defines:

\begin{eqnarray}
\q(H)&=&\oplus_{\substack{\tau\in\Delta(\mathfrak{g},\mathfrak{t})\\\tau(H)\geq
0}}\ \mathfrak{g}^\tau,\\
\l(H)&=&\oplus_{\substack{\tau\in\Delta(\mathfrak{g},\mathfrak{t})\\\tau(H)=0
}}\ \mathfrak{g}^\tau\text{ and }\\
\u(H)&=&\oplus_{\substack{\tau\in\Delta(\mathfrak{g},\mathfrak{t})\\\tau(H)>0
}}\ \mathfrak{g}^\tau.
\end{eqnarray}
Then $\q(H)$ is a parabolic subalgebra of $\g$ and $\q(H)=\l(H)\oplus \u(H)$ is a Levi
decomposition. Since $\l(H)$ is defined over $\R$, there exists a
well defined subgroup $L(H)$ of $G$ of complexified Lie algebra $\l(H)$.
\begin{defi}\label{defi2.1}
A pair $(\q(H),L(H))$ defined by an element $H\in\imath\t_0$ is called a
theta stable parabolic algebra.
\end{defi}
Let $(\q,L)$ be a parabolic theta stable algebra of $\g$. Let $\u$ be the radical unipotent of $\g$. One defines:
\begin{eqnarray}\label{eq2.2}
R(\mathfrak{q})=\dim\mathfrak{p}\cap\mathfrak{u}.
\end{eqnarray}
the cohomological degree of $\q$.
By \cite{VZ}, $\bigwedge^{R(\mathfrak{q})}\left(\mathfrak{p}\cap\mathfrak{u}\right)$ is the
highest weight vector in $\bigwedge^{R(\q)}\mathfrak{p}$. Let
$V(\mathfrak{q})$ be the irreducible $K$-submodule of
$\bigwedge^{R(\mathfrak{q})}\mathfrak{p}$ generated by this vector. These
modules play an important role in the study of the cohomology of locally symmetric spaces. We are going
to explain their classification up to isomorphism (this is done for example by Bergeron in \cite{B1}). Clearly, if two theta stable parabolic algebras are
$K$-conjugated they generate the same module. So up to $K$-conjugation, it
can be assumed that $\q$ is defined by an element: 
$$H=(x_1,\dots,x_a)\otimes( y_1,\dots,y_b)\in \R^{a}\times\R^b$$
with:
$$x_1\geq\dots\geq x_a\text{ and }y_1\geq\dots\geq y_b.$$
Such an element will be called dominant.
\begin{defi}\label{defi2.2}
 A parabolic theta stable algebra defined by a dominant element is
  called standard.
\end{defi}
We are going to associate to $H$ two partitions. Recall that a partition is
a decreasing sequence $\alpha$ of natural integers
$\alpha_1,\dots,\alpha_l\geq 0$. The Young diagram of $\alpha$, simply written $\alpha$, is obtained by adding from top to bottom rows of $\alpha_i$ squares all of
the same shape. For example the Young diagram associated to the partition $(3,2,1)$ is:
$$\yng(3,2,1).$$
Let $\alpha$ and $\beta$ be partitions such that the diagram of $\alpha$
is included in the diagram of $\beta$, written $\alpha\subset\beta$. We
will also write $\beta\bk\alpha$ for the
complementary of the diagram of $\alpha$ in the diagram of $\beta$. It is
a skew diagram. The notations $a\times b$ or $b^a$ stand for the partition
$$(\underbrace{b,\dots,b}_{a\text{ times}}).$$ 
For a diagram $\alpha$, the diagram $\widetilde{\alpha}$ is defined as the diagram counting the length of
columns. For example:
$$\widetilde{(3,2,1)}=(3,2,1).$$
For an inclusion of diagram $\alpha\subset a\times b$, the
diagram $\widehat{\alpha}$ is defined as $[(a\times b)\bk\alpha]$ rotated by an angle of
$\pi$, it is a partition. For example, the inclusion $(3,2,1)\subset 3\times 3$ gives
$\widehat{(3,2,1)}=(2,1)$.
We remark that for $\alpha\subset a\times b$:
$$\widehat{\widetilde{\alpha}}=\widetilde{\widehat{\alpha}}.$$
Let $H\in\imath\t_0^*$ be dominant. One associates to $H$ two partitions $\alpha\subset\beta\subset a\times b$ defined by:
\begin{eqnarray*}
\alpha(i)&=&|\{j|x_i> y_{b+1-j}\}|\text{ and }\\
\beta(i)&=& |\{j|x_i\geq y_{b+1-j}\}|.
\end{eqnarray*}

\begin{defi}\label{defi2.3}
If $\t$ acts on a vector space $V$, one writes $\rho(V)$ for half the sum
of weights of $\t$ on $V$. It is an element of $\t^*$ that can be seen as
an element of $\C^{a}\times\C^{b}$.   
\end{defi}

Let $\q$ be a parabolic theta stable algebra and $(\alpha,\beta)$ be the
associated partitions. By definition of $(\alpha,\beta)$, one has:
\begin{eqnarray}\label{eq2.3}
\Delta(\u\cap\p)=\{x_i-y_{b+1-j}|\ (i,j)\in\alpha\}\cup \{-(x_i-y_{b+1-j})|\ (i,j)\notin\beta\}.
\end{eqnarray}
Using this equality, it is a simple calculation to see that the highest weight of $V(\q)$ is:
\begin{eqnarray}\label{eq2.4}
2\rho(\u\cap\p)=(\alpha_i+\beta_i-b)_{1\leq i\leq
  a}\otimes(a-(\widetilde{\alpha}_{b+1-j}+\widetilde{\beta}_{b+1-j}))_{1\leq
  j\leq b}.
\end{eqnarray}
So $V(\q)\simeq V(\q')$ if $\q$ and $\q'$ have the same associated partitions. Indeed:

\begin{prop}\label{prop2.1}  The following three points give the classification of modules $V(\q)$:
\begin{itemize}
\item let $\q$ be a parabolic theta stable algebra and $\alpha\subset \beta$ be the associated partitions, then $(\beta\bk\alpha)$ is a union of squares which intersect only on vertices.

\item we have $V(\q)\simeq V(\q')$ if and only if $(\q,L)$ and $(\q',L')$ have the same associated
  partitions.
\item If $\alpha\subset\beta\subset a\times b$ is a pair of partitions verifying the
  condition of the first point there exists a parabolic theta stable algebra $\q$
  with associated partition $(\alpha,\beta)$. Such a pair will be called compatible.
\end{itemize}
\end{prop}

\subsection{Representations}\label{ssect2.2}

Let $H\in \imath\mathfrak{t}_0$ and $(\q, L)$ be the associated
theta stable parabolic algebra. Choosing a positive set of roots $\Delta^+(\l)$ in $\Delta(\l)$ one can
define:
$$\rho_\q=\rho(\u)+\rho^+(\l),$$
where $\rho^+(\l)$ is half the sum of the roots in $\Delta^+(\l)$.
\begin{theo}\label{theo2.0}
 Let
$\lambda$ be the differential of a unitary character of $L$ viewed as an
element of $\mathfrak{t}^*$ such that:
\begin{equation}\label{eq2.5}
(\lambda,\tau)\geq 0\ \forall
\tau\in\Delta(\mathfrak{u}),
\end{equation}
then:
\begin{itemize}
\item  there exists a unique irreducible and unitarisable $(\g,K)$-module, which will
be denoted  $A_\mathfrak{q}(\lambda)$, verifying the two following
properties :
\begin{itemize}
\item the infinitesimal character of $A_\q(\lambda)$ is
  $\lambda+\rho_\mathfrak{q}$ and 
\item the $K$-type
  $\lambda+2\rho(\mathfrak{u}\cap\mathfrak{p})$ appears in $A_\q(\lambda)$.
\end{itemize}
\item Furthermore, every $K$-type of $A_\q(\lambda)$ is of the form:
\begin{equation}\label{eq2.6}
\lambda+2\rho(\u\cap\p)+\sum_{\beta\in\Delta(\u\cap\p)}n_\tau \tau
\end{equation}
with $n_\tau\geq 0$.
\end{itemize}
\end{theo}
The construction of $A_\q(\lambda)$ is done using cohomological induction
(see \cite{KV}). The condition (\ref{eq2.5}) can be relaxed for the
construction of $A_\q(\lambda)$ but it is necessary  to have at the same time
irreducibility, unitarity and unicity. The case where $\lambda=0$ is linked
to cohomology of locally symmetric spaces, see equation (\ref{eq4.1}). 
\begin{rema}
The representation $A_{\q}(\lambda)$ is a discrete series if and only if $L$ is compact.
\end{rema}
In the next paragraph, packets for cohomologically induced representation
following Arthur \cite{Arthur} and Adams and Johnson \cite{AJ} are defined.

\subsection{Packets for cohomological representations}\label{ssect2.3}

Let
$$H=(x_1,\dots,x_a)\otimes (y_1,\dots,y_b)$$
be dominant. Let $(\q,L)$ be the
associated parabolic theta stable algebra. Let $\lambda$ be the differential
of a character of $L$ verifying (\ref{eq2.5}). If $w\in W(\mathfrak{g},\mathfrak{t})$ (the Weyl group of $\t$ in $\g$) then $w\bullet\mathfrak{q}$ is a parabolic theta stable algebra and if $\lambda$ verifies (\ref{eq2.5})
then $w\bullet \lambda$ verify (\ref{eq2.5}) for  $w\bullet \mathfrak{u}$. Adams and Johnson in \cite{AJ} suggest that the following packet: 
\begin{equation}\label{eq2.6.1}
\Pi_{\lambda,\q}=\{A_{w\bullet\mathfrak{q}}(w\bullet\lambda)|\ w\in
W(\g,\t)\}
\end{equation}
is an Arthur packet. Following Arthur (\cite{Arthur}), we explain the  construction of the Arthur parameter that should be associated to this packet. 
\begin{defi}
A $L$-homomorphism $\psi$ between $\ ^LL$ and $\ ^LG$ is a morphism of groups between
$^LL$ and $^LG$ such that the following diagram commutes:
$$\xymatrix{ \ ^LL \ar[rr]^{\psi} \ar[rd] &&\ ^LG \ar[ld] \\ & W_\R}.$$
\end{defi}
We want to built such a morphism. Let
$z_1,\dots,z_r$ be the different values of the $(x_i)$ and the $(y_j)$. Let $(a_i)_{1\leq i\leq r}$ and $(b_i)_{1\leq i\leq r}$ be the integers verifying:
\begin{equation}\label{eq2.6.2}
\begin{split}
(x_1,\dots,x_a)&=\left((\underbrace{z_1,\dots,z_1}_{a_1\text{ times}}),\dots,(\underbrace{z_r,\dots,z_r}_{a_r\text{ times}})\right)\text{ and}\\
(y_1,\dots,y_b)&=\left(\underbrace{z_1,\dots,z_1}_{b_1},\dots,\underbrace{z_r,\dots,z_r}_{b_r}\right).
\end{split}
\end{equation}
Then we can see that:
$$L=\prod_{i=1}^{r}U(a_i,b_i).$$
Let $n_i=a_i+b_i$ and $(k_i)_{1\leq i\leq r}\in\Z_{\geq 0}$. We define:

\begin{equation}\label{eq2.7}
\begin{split}
\psi_\q(g_1,\dots,g_r)&=\left(\begin{array}{ccc}g_1&&\\&\ddots&\\&&g_r\end{array}\right),\\
\psi_\q(z)&=\left(\begin{array}{ccc}\mu(z)^\frac{k_1}{2}I_{n_1}&&\\&\ddots&\\&&\mu(z)^{\frac{k_r}{2}}I_{n_r}\end{array}\right)\text{ and}\\
\psi_\q(j)&=\left(\begin{array}{ccc}\Phi_{n_1}&&\\&\ddots&\\&&\Phi_{n_r}\end{array}\right)\Phi_n^{-1}.
\end{split}
\end{equation}
One can check that $\psi_\q$ extends itself to an $L$-homomorphism: 
$$\ ^LL=^L(U(n_1)\times\dots\times U(n_r))\stackrel{\psi_\q}{\rightarrow}\
^LU(n)=\ ^LG,$$
 if and only if:
\begin{equation}\label{eq2.8}
k_i\equiv n-n_i\ [2].
\end{equation}
Let
$$m_i=-n_1-\dots-n_{i-1}+n_{i+1}+\dots+n_r$$
then clearly the $(m_i)$ satisfy the relation (\ref{eq2.8}) and $\psi_\q$
is defined by the formulas  $(\ref{eq2.7})$ with $k_i=m_i$.
\begin{defi}\label{defi2.4}
Let
$$\psi_\lambda:W_\R\times\text{SL}(2,\C)\rightarrow \ ^LL,$$
be the Arthur parameter associated to the character $\lambda$ (see example \ref{exem1.1}). We set:
$$\psi_{\lambda,\q}=\psi_\q\circ\psi_\lambda.$$
It should be the parameter associated to $\Pi_{\lambda,\q}$.
\end{defi}

Recall that the infinitesimal character of a parameter $\psi$ has been
defined (definition \ref{defi1.4}).
\begin{lemm}\label{lemm2.1}
 Every element of $\Pi_{\lambda,\q}$ has the infinitesimal character of $\psi_{\lambda,\q}$.
\end{lemm}

\begin{proof}
As $\lambda$ is the differential of a character of $L$:
$$\lambda=\left(\underbrace{\lambda_1,\dots,\lambda_1}_{a_1},\dots,\underbrace{\lambda_r,\dots,\lambda_r}_{a_r}\right)\otimes\left(\underbrace{\lambda_1,\dots,\lambda_1}_{b_1},\dots,\underbrace{\lambda_r,\dots,\lambda_r}_{b_r}\right)$$
with $\lambda_i\in\Z$. The condition (\ref{eq2.5}) can be written as:
\begin{eqnarray*}\label{eq2.9}
\lambda_i-\lambda_{i+1}\geq 0.
\end{eqnarray*}
Since
$${\psi_\lambda}_{|W_\C\times\text{SL}(2,\C)}=\oplus_{i=1}^r\mu^{\lambda_i}\otimes\sigma_{n_i},$$
it comes that:
\begin{equation}\label{eq2.10}
{\psi_{\lambda,\q}}_{|W_\C\times \text{SL}(2,\C)}=\oplus_{i=1}^r
\mu^{\lambda_i+\frac{m_i}{2}}\otimes\sigma_{n_i}.
\end{equation}
Using definition \ref{defi1.4}, we find $\lambda+\rho_\q$. 
\end{proof}

\section{Archimedean theta correspondance and Adams conjecture}\label{sect3}

\subsection{Definition}\label{defi3.1}     
 Let $V'$ (resp. $V$) be a $\C$-vector space with a non degenerate hermitian
(resp. skew hermitian) form
$(,)'$ (resp. $(,)$). Let ${W}=V'\otimes_\C V$. With the form

$$B(,)=\Re\left( (,)'\otimes(,)^\iota\right),$$ 
the $F$-vector space $W$ becomes a non degenerate symplectic space.
Let:
$$G'=U(V')\text{ and }G=U(V).$$
The pair of group: 
$$G'\times G\subset \text{Sp}(W)$$
is a dual pair in the sense of Howe \cite{Howe}. Let $\mathbb{K}$ be a maximal compact subgroup of $\text{Sp}(W)$ such that $G\cap \mathbb{K}$ and $G'\cap\mathbb{K} $ are maximal compact subgroups of $G$ respectively $G'$.
Let ${\text{Mp}({W})}$ be the metaplectic covering of
$\text{Sp}({W})$, it is an extension :
$$1\rightarrow \S^1\rightarrow \text{Mp}({W})\rightarrow
\text{Sp}({W})\rightarrow 1.$$
The metaplectic representation $(\omega, \mathbf{S})$ is a unitary representation of the metaplectic group (see \cite{Howe} for a definition).
 Let $n'$ (resp. $n$) be the
dimension of $V'$ (resp. $V$). Let us choose a pair of unitary characters
$\chi=(\chi_1,\chi_2)$ of $\C^\times$.  Let $\alpha(\chi_1)$ and $\alpha(\chi_2)$ be the integers such that:
$$\chi_i(u)=u^{\alpha(\chi_i)}\ \forall u\in\S^1,$$
these integers determine $\chi$. Let $\epsilon$ be the $\text{sign}$ character of $\R^\times$. We assume that:
$${\chi_1}_{|\R^\times}=\epsilon^{n}\text{ and }{\chi_2}_{|\R^\times}=\epsilon^{n'}.$$
According to Kudla \cite{K} this choice of characters determines a splitting:
$$\iota_\chi:G'\times G\rightarrow \text{Mp}(W),$$
and by restriction of the metaplectic representation of $\text{Mp}(W)$ a
representation $(\omega_\chi,\mathbf{S})$ of $G'\times G$. Let
$\widetilde{\mathbb{K}}$ be the inverse image of $\mathbb{K}$ in the
metaplectic group and $S$ be the $\widetilde{\mathbb{K}}$- finite vectors
of $\mathbf{S}$. We will write $S_\chi$ for the representation of $(\g\times\g',K'\times K)$ on $S$ given by the choice of splitting $\iota_\chi$. The archimedean theta correspondance is defined using the following theorem of Howe \cite{Howe}.
\begin{theo}\label{theo3.1}
Let $\pi'$ be an irreducible  $(\g',K')$-module. There exists at most one irreducible $(\g,K)$-module such that:
$$\hom_{\g'\times\g,K'\times K}(S_\chi,\pi'\otimes\pi)\neq 0,$$
if such a $\pi$ exists, we will denote it $\theta_\chi(\pi')$ and otherwise we will set $\theta_\chi(\pi)=0.$ 
\end{theo}
The $(\g,K)$-module $\theta_\chi(\pi)$ is called the theta lift of
$\pi'$. To compute the theta lift of a representation $\pi'$, one has first to
prove that theta lifts are non zero. We recall two theorems of Jian-Shu Li.

\begin{defi}\label{defi3.2} The dual pair $(G',G)$ is said to be in the stable range if $V$ has an
isotropic subspace of dimension $n'$.
\end{defi}

\begin{theo}\label{theo3.2}\cite{Li2,Li3}
Suppose that:
\begin{itemize} 
\item $\dim V'\leq \dim V$ (resp. $(G',G)$ is in the stable range) and 
\item $\pi'$ is a discrete series (resp. unitarisable).
\end{itemize}
Then $\theta(\pi',V)$ is non zero and unitarisable.
\end{theo}

\subsection{Arthur packets and theta correspondance}\label{thetapsi}

We consider the dual pair $(G',G)$ with:
$$G'=U(a',b')\text{ and } G={U(a,b)}.$$
Let $n'=a'+b'$ and $ n=a+b$. We suppose that $n'\leq n$. Let $\psi'$ be an
Arthur parameter for $G'$.  We
associate to $\psi'$ an Arthur parameter $\theta_\chi(\psi)$ for the group
$G$ defined by:
\begin{eqnarray}\label{eq3.1}
 \theta_\chi(\psi')_{|\C^\times\times\SL}=\mu^{\frac{\alpha(\chi_2)-\alpha(\chi_1)}{2}}\otimes\psi_{|\C^\times\times\SL}\oplus\mu^{\frac{\alpha(\chi_2)}{2}}\otimes\sigma_{n-n'}
\end{eqnarray}
and
\begin{equation}\label{eq3.1.1}
\theta_\chi(\psi')(j)=\left(\begin{array}{cc}\psi'(j)\Phi_{n'}&0\\0&\Phi_{n-n'}\end{array}\right)\Phi_n^{-1}\times
j.
\end{equation}
By \cite{HKS} or \cite{Adams} it is a well defined Arthur parameter.

Adams made the following two conjectures about elements of $\Pi^G(\theta_\chi(\psi'))$.

\begin{conj}\label{conj1}
If $\pi'\in\Pi^{U(a',b')}(\psi')$ then $\theta_\chi(\pi')\in \Pi^{U(a,b)}(\theta_\chi(\psi'))$.
\end{conj}

\begin{conj} \label{conj2}
In the stable range (i.e $a'+b'\leq\min(a,b)$) we have:
$$\Pi^{U(a,b)}(\theta_\chi(\psi'))=\cup_{a'+b'=n'}\theta_\chi(\Pi^{U(a',b')}(\psi')).$$
\end{conj}
We use the notations of paragraph \ref{ssect2.3}. Therefore the notations
$H,\q,L,\lambda,x_i,y_i,\lambda_i,r,z_i,a_i,b_i$ are well defined. We want to
know when the parameter $\psi_{\q,\lambda}$ can be of the form
$\theta_\chi(\psi')$ for some $\chi,a',b',\psi'$. According to equation (\ref{eq2.10}) and (\ref{eq3.1})
there must exist $r_0\in\{1,\dots,r\}$ such that:

\begin{equation}\label{eq3.2}
n_{r_0}=a+b-(a'+b').
\end{equation}

We choose $r_0$, then:
\begin{multline}\label{eq3.3}
{\psi_{\lambda,\q}}_{|\C^\times\times\SL}=\\
\mu^{\lambda_{r_0}+\frac{m_{r_0}-\alpha(\chi_2)}{2}}\otimes\left( \mu^{\frac{\alpha(\chi_2)}{2}}\otimes\sigma_{n_{r_0}}\oplus\mu^{\frac{\alpha(\chi_2)-\alpha(\chi_1)}{2}}\left(\oplus_{i\neq
r_0}
\mu^{\lambda_i-\lambda_{r_0}+\frac{m_i-m_{r_0}+\alpha(\chi_1)}{2}}\otimes\sigma_{n_i}\right)\right).$$
\end{multline}
Define:
$$a'=a_1+\dots+a_{r_0-1}+b_{r_0+1}+\dots+b_r,$$
and:
$$b'=b_1+\dots+b_{r_0-1}+a_{r_0+1}+\dots+a_r.$$
This choice of $a',b'$ satisfies (\ref{eq3.2}). We define an element
$$H'=(x_1',\dots,x_{a'}')\otimes(y_1',\dots,y_{b'}')$$
of the Cartan algebra of $\u(a',b')$ by the equalities:
$$(x_1',\dots,x_{a'}')=\left(\underbrace{z_1,\dots,z_1}_{a_1\text{ times}},\dots,\underbrace{z_{r_0-1},\dots,z_{r_0-1}}_{a_{r_0-1}\text{ times}},\underbrace{z_{r_0+1},\dots,z_{r_0+1}}_{b_{r_0+1}\text{ times}},\dots,\underbrace{z_r,\dots,z_r}_{b_r\text{ times}}\right)$$
and
$$(y_1',\dots,y_{b'}')=\left(\underbrace{z_1,\dots,z_1}_{b_1\text{ times}},\dots,\underbrace{z_{r_0-1},\dots,z_{r_0-1}}_{b_{r_0-1}\text{ times}},\underbrace{z_{r_0+1},\dots,z_{r_0+1}}_{a_{r_0+1}\text{ times}},\dots,\underbrace{z_{r},\dots,z_r}_{a_{r}\text{ times}}\right).$$

\begin{defi}\label{defi3.2.1}
Let $\q$ (resp. $\q'$) be the theta parabolic algebra of $\g$ (resp. $\g'$) obtained from $H$ (resp. $H'$). We will write this relation $\q=\theta(\q')$.
\end{defi}
The objects associated to
   $H'$ according to the definitions of paragraph \ref{ssect2.3} are
   designated by the same notation plus a prime. We have that:
\begin{equation}\label{eq3.4}
\begin{split}
r'&=r-1,\\
L'&=\prod_{i=1}^{r_0-1}U(a_i,b_i)\times\prod_{i=r_0+1}^{r}U(b_i,a_i),\\
n_j'&=n_j\ (j\neq r_0),\\
m_j-m'_j&=n_{r_0}\text{ if }j<r_0\text{ and}\\
m_j-m'_j&=-n_{r_0}\text{ if }j>r_0.\\
\end{split}
\end{equation}

So:
\begin{equation*}
\begin{split}
{\psi}_{|W_\C\times\text{SL}(2,\C)}=&\mu^{\lambda_{r_0}+\frac{m_{r_0}-\alpha(\chi_2)}{2}}\otimes\left(
  \mu^{\frac{\alpha(\chi_2)}{2}}\sigma_{n_{r_0}}\oplus\mu^{\frac{\alpha(\chi_2)-\alpha(\chi_1)}{2}}\right.\\
    &\left.\left(\oplus_{i<
r_0}
\mu^{\lambda_i-\lambda_{r_0}+\frac{m_i'+ n_{r_0}-m_{r_0}+\alpha(\chi_1)}{2}}\otimes\sigma_{n_i}\oplus_{i>
r_0}
\mu^{\lambda_i-\lambda_{r_0}+\frac{m_i'-n_{r_0}-m_{r_0}+\alpha(\chi_1)}{2}}\otimes\sigma_{n_i}\right)\right).
\end{split}
\end{equation*}
Adams conjecture suggests therefore to define:
$$
\lambda_i'=
\left\{\begin{array}{cc}
\lambda_i-\lambda_{r_0}+\frac{n_{r_0}-m_{r_0}+\alpha(\chi_1)}{2}&\text{ if }i<r_0\\
\lambda_i-\lambda_{r_0}+\frac{-n_{r_0}-m_{r_0}+\alpha(\chi_1)}{2}&\text{ if }i>r_0.
\end{array}
\right.
$$
The $(\lambda_i')$ satisfy the conditions (\ref{eq2.5}) and are in
$\Z$. So the $(\g',K')$-module $A_{\q'}(\lambda')$ is well defined irreducible and unitarisable. 
We state our main theorem and some corollaries.
\begin{theo}\label{theo3.3} We have that:
\begin{itemize}
\item $\theta_\chi(\psi_{\q',\lambda'})=\psi_{\q,\lambda-(\lambda_{r_0}+\frac{m_{r_0}-\alpha(\chi_2)}{2})}$.
\item If $\theta_\chi(A_{\q'}(\lambda'))$ is non zero and unitary then:
$$\left(\det\bullet\right)^{\lambda_{r_0}+\frac{m_{r_0}-\alpha(\chi_2)}{2}}\otimes\theta_\chi(A_{\q'}(\lambda'))=A_\q(\lambda).$$
\end{itemize}
\end{theo}

\begin{coro}\label{coro3.1} 
We have that:
\begin{itemize} 
\item  if $L'$ is compact then
$$\left(\det\bullet\right)^{\lambda_{r_0}+\frac{m_{r_0}-\alpha(\chi_2)}{2}}\otimes\theta_\chi(A_{\q'}(\lambda'))=A_\q(\lambda),$$
\item if $a'+b'\leq \min(a,b)$ (in the stable range):
$$
\left\{
\begin{array}{ll}
\left(\det\bullet\right)^{\lambda_{r_0}+\frac{m_{r_0}-\alpha(\chi_2)}{2}}\otimes\theta_\chi(A_{\q'}(\lambda'))=A_\q(\lambda)\text{ and}\\$$

\Pi^{U(a,b)}(\psi_{\q,\lambda-(\lambda_{r_0}+\frac{m_{r_0}-\alpha(\chi_2)}{2})})=\cup_{a'+b'=n-n'}\theta_\chi(\Pi^{U(a',b')}(\psi_{\q',\lambda'}))
\end{array}
\right..
$$

\end{itemize}
\end{coro}

\begin{proof}
We deduce the corollary from the theorem. The first two equalities are 
consequences of the second point of theorem \ref{theo3.3} and of theorem
\ref{theo3.2}. The last equality is an easy consequence of
the second point of the corollary and of the explicit description of
packets $\Pi_{\lambda,\q}$ given by equation (\ref{eq2.6.1}).

\end{proof}

\begin{proof}[Beginning of the proof of theorem \ref{theo3.3}]

According to the discussion just before the theorem, we have choosen $r_0$ and integers $(\lambda'_i)$ to have:
$${\theta_\chi(\psi_{\q',\lambda'})}_{|W_\C\times\text{SL}(2,\C)}={\psi_{\q,\lambda-(\lambda_{r_0}+\frac{m_{r_0}-\alpha(\chi_2)}{2})}}_{|W_\C\times\text{SL}(2,\C)},$$
the extension of the equality to $W_\R$ is just a verification using the formulas for
the action of $j$ (equations (\ref{eq2.7}) and (\ref{eq3.1.1})).

We explain the scheme of the proof of the second point, our main result. Since we assume
that $\lambda$ verifies (\ref{eq2.5}) and that the theta lift is non zero and
unitary by theorem \ref{theo2.0}, it is enough to prove that the $(\g,K)$-module
 $$\left(\det\bullet\right)^{\lambda_{r_0}+\frac{m_{r_0}-\alpha(\chi_2)}{2}}\otimes\theta_\chi(A_{\q'}(\lambda'))$$
\begin{itemize}
\item 
  has infinitesimal character $\lambda+\rho_\q$ and
\item contains the $K$-type of highest weight $\lambda+2\rho(\u\cap\p)$.
\end{itemize}
We prove these two points in the next section (corollary \ref{coro3.3} and
corollary \ref{coro3.4}).

\end{proof}
An other reformulation of the corollary is:
\begin{coro}\label{coro3.2} Let $(\q,L)$ a standard parabolic theta stable algebra with:
$$L=\prod_{i=1}^rU(a_i,b_i)$$
and $(\lambda_i)\in \Z^{a+b}$ such that $\lambda_i\geq\lambda_{i+1}$. If:
$$\max(a_i+b_i)\geq\max(a,b)$$
then up to a character $A_\q(\lambda)$ is a theta lift from a
representation $A_\q'(\lambda')$ such that $\lambda'$ verifies the condition $(\ref{eq2.5})$ and 
$$L'=\prod_{i=1}^{r_0-1}U(a_i,b_i)\times\prod_{i=r_0+1}^{r}U(b_i,a_i).$$
\end{coro}
\subsection{Proof of theorem \ref{theo3.3}}

\subsection{Types}
We identify $K'$-representations with their highest weights. These
are exactly the dominant elements of $\Z^{a'}\times\Z^{b'}$. We write a weight $\mu'$ in the following way
$$\mu'=\frac{\alpha(\chi_1)}{2}+\left\{(x_1',\dots,x_a')\otimes(y_1',\dots,y_b')\right\}$$
with the $(x_i')$ and $(y_i')$ in $\frac{a+b}{2}+\Z$. We define for $\mu'$:
$$\text{deg}(\mu')=\sum_i|x_i-\frac{a-b}{2}|+\sum_i|y_i-\frac{b-a}{2}|.$$
The following theorem describes parts of the relationship between theta
correspondance and types.
\begin{theo}[Howe]\label{theo3.4}
Let $\pi'$ be an irreducible $(\g',K')$-module such that
$\pi=\theta_\chi(\pi')$ is non zero. Let $\mu'$ be a $K'$-type of $\pi'$ of minimal degree. Then if we write:
\begin{equation}\label{mu}
\begin{split}
\mu'=&\frac{\alpha(\chi_1)}{2}+\\
&\left\{\frac{a-b}{2}+(a_1,\dots,a_t,0,\dots,0,-b_1,\dots,-b_u)\otimes \frac{b-a}{2}+(c_1,\dots,c_v,0,\dots,0,-d_1,\dots,-d_w)\right\}
\end{split}
\end{equation}
with $a_i,b_i,c_i,d_i>0$ then:
\begin{itemize}
\item $t+w\leq a$ and $v+u\leq b$.
\item 
\begin{equation}\label{mu'}
\begin{split}
\theta(\mu')=&\frac{\alpha(\chi_2)}{2}+\\
&\left\{\frac{a'-b'}{2}+(a_1,\dots,a_t,0,\dots,0,-d_1,\dots,-d_w)\otimes \frac{b'-a'}{2}+(c_1,\dots,c_v,0,\dots,0,-b_1,\dots,-b_u)\right\}
\end{split}
\end{equation}
is a ${K}$-type of $\pi$.
\end{itemize}
\end{theo}
Recall that in paragraph \ref{ssect2.1}, we have associated to $H$ a pair of Young
 diagrams. Let

\begin{align}
m&=b_{1}+\dots+b_{r_0-1}&s&=b_{r_0+1}+\dots+b_r\\
k&=a_1+\dots+a_{r_0-1}&l&=a_{r_0+1}+\dots+a_r.
\end{align}
Then because of equation (\ref{eq2.6.2}), the pair of diagrams $(\alpha,\beta)$ 
associated to $H$ has the following shape:


\begin{center}
\includegraphics[width=7cm, height=4cm]{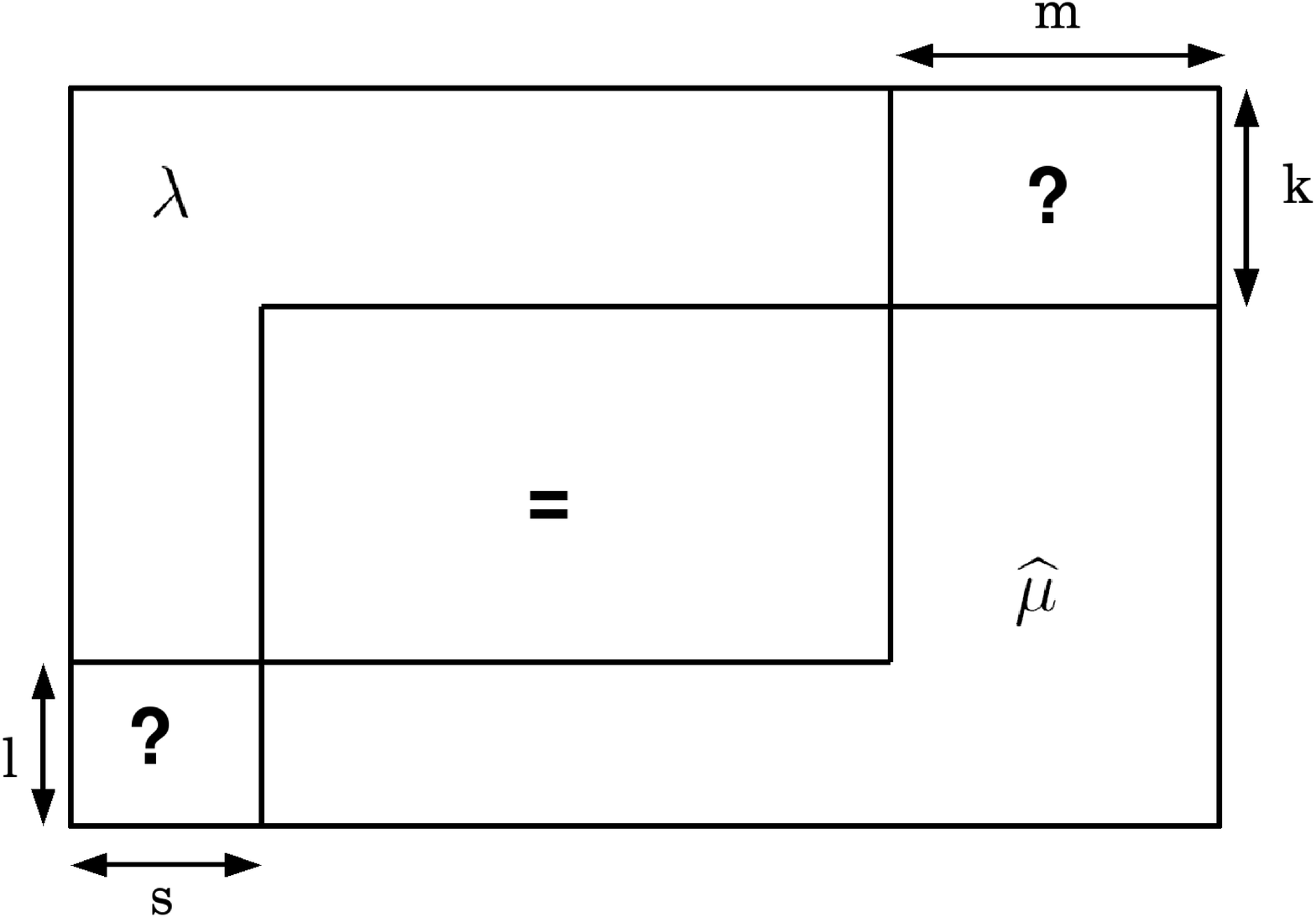}
\end{center}

\begin{itemize}
\item In the upper box with a question mark the pair of partitions $(\alpha,\beta )$ is
equal to
$$\alpha_1\subset\beta_1\subset k\times m\text{ and }$$ 
\item  In the lower box with a question mark the pair of partitions $(\alpha,\beta )$ is
equal to $\alpha_2\subset\beta_2\subset l\times s$.
\end{itemize}
The two pairs of partitions $(\alpha_i,\beta_i)$ $i=1,2$ are compatible
pairs of partitions and the pair of diagrams $(\alpha',\beta')$ associated to $H'$ is:
$$\begin{array}{|c|c|} \hline
+&\alpha_1\subset\beta_1\\ \hline
\widetilde{\widehat{\beta_2}}\subset \widetilde{\widehat{\alpha_2}}&-\\ \hline
\end{array}
$$
According to equation (\ref{eq2.4}):
\begin{equation}
\begin{split}
2\rho(\u\cap\p)&=(\underbrace{\alpha_1(i)+\beta_1(i)+b-2m}_{1\leq i\leq
k},\underbrace{s-m,\dots,s-m}_{a_{r_0}\text{ times}},\underbrace{\alpha_2(i)+\beta_2(i)-b}_{1\leq
i\leq l})\\
&\otimes(\underbrace{a-\widetilde{\alpha_1}(m+1-i)-\widetilde{\beta_1}(m+1-i)}_{1\leq
i\leq m},\underbrace{l-k,\dots,l-k}_{b_{r_0}\text{ times}},\underbrace{2l-a-\widetilde{\alpha_2}(s+1-i)-\widetilde{\beta_1}(s+1-i)}_{1\leq
i\leq s})
\end{split}
\end{equation}
and:
\begin{equation}
\begin{split}
2\rho(\u'\cap\p')&=(\underbrace{\alpha_1(i)+\beta_1(i)+l-m}_{1\leq i\leq
k},\underbrace{l-m-\widetilde{\alpha_2}(s+1-i)-\widetilde{\beta_2}(s+1-i)}_{1\leq
i\leq s})\\
&\otimes(\underbrace{k+s-\widetilde{\alpha_1}(m+1-i)-\widetilde{\beta_1}(m+1-i)}_{1\leq
i\leq m},\underbrace{\alpha_2(i)+\beta_2(i)-a'}_{1\leq
i\leq l}).
\end{split}
\end{equation}
\begin{lemm} We have:
\begin{itemize}
\item that the vector $\lambda'+2\rho(\u'\cap\p')$ is of the form $(\ref{mu})$ with
  $t=k,u=s,v=m$ and $w=l$,
\item the equality
  $\lambda_{r_0}+\frac{m_{r_0}-\alpha(\chi_2)}{2}+\theta_\chi(\lambda'+2\rho(\u'\cap\p'))=\lambda+2\rho(\u\cap\p)$ and
\item that the vector $\lambda'+2\rho(\u'\cap\p')$ is of minimal degree in $A_{\q'}(\lambda')$.
\end{itemize}
\end{lemm}
\begin{proof} 

We prove the first point. Using that:
\begin{equation}
\begin{split}
n_{r_0}&=a+b-(l+s+k+m)\text{ and}\\
m_{r_0}&=-(k+m)+(l+s).
\end{split}
\end{equation}
it comes that:
\begin{equation}
\begin{split}
&\lambda'+2\rho(\u'\cap\p')-(\frac{\alpha(\chi_1)}{2})=(\underbrace{\lambda_i-\lambda_{r_0},\dots,\lambda_i-\lambda_{r_0}}_{i<r_0\
a_i\text{ times}},\underbrace{\lambda_i-\lambda_{r_0},\dots,\lambda_i-\lambda_{r_0}}_{i>r_0\
b_i\text{ times}})\\
&\otimes(\underbrace{\lambda_i-\lambda_{r_0},\dots,\lambda_i-\lambda_{r_0}}_{i<r_0\
b_i\text{ times}},\underbrace{\lambda_i-\lambda_{r_0},\dots,\lambda_i-\lambda_{r_0}}_{i>r_0\
a_i\text{ times}})\\
&+\left[\frac{a-b}{2}+\left(\underbrace{\alpha_1(i)+\beta_1(i)+b-(s+m)}_{1\leq i\leq
k},\underbrace{-(a-(k+l))-\widetilde{\alpha_2}(s+1-i)-\widetilde{\beta_2}(s+1-i)}_{1\leq
i\leq s}\right)\right]\\
&\otimes\left[\frac{b-a}{2}+\left(\underbrace{a-(k+l)+(k-\widetilde{\alpha_1}(m+1-i))+(k-\widetilde{\beta_1}(m+1-i))}_{1\leq
i\leq m}\right.\right.\\
&\left.\left.,(\underbrace{\alpha_2(i)-s)+(\beta_2(i)-s)-(b-(m+s))}_{1\leq
i\leq l}\right)\right].
\end{split}
\end{equation}
So the first point is proved.

Let us prove the second point. Using equation (\ref{mu'}), we see that:
\begin{equation}
\begin{split}
&\left(\lambda_{r_0}+\frac{m_{r_0}-\alpha(\chi_2)}{2}\right)+\theta(\lambda'+2\rho(\u'\cap\p'))=\lambda+(\frac{m_{r_0}}{2})\\
&+\frac{a'-b'}{2}+((\underbrace{\alpha_1(i)+\beta_1(i)+b-(s+m)}_{1\leq i\leq
k},\underbrace{0,\dots,0}_{a-(k+l)},\underbrace{\alpha_2(i)+\beta_2(i)-b+m-s}_{1\leq
i\leq l})\\
&\otimes\frac{b'-a'}{2}+(\underbrace{a+k-l-\widetilde{\alpha_1}(m+1-i))-\widetilde{\beta_1}(m+1-i)}_{1\leq
i\leq m},\underbrace{0,\dots,0}_{b-(m+s)},\\
&\underbrace{-a+k+l-\widetilde{\alpha_2}(s+1-i)-\widetilde{\beta_2}(s+1-i)}_{1\leq
i\leq s})).\\
\end{split}
\end{equation}
Using the following two computations:
\begin{equation}
\begin{split}
\frac{a'-b'}{2}+\frac{m_{r_0}}{2}&=s-m\text{ and}\\
\frac{b'-a'}{2}+\frac{m_{r_0}}{2}&=l-k,
\end{split}
\end{equation}
the second point is found..

The last point is an easy consequence of the description of
${K'}$-types of $A_{\q'}(\lambda')$ given by equation (\ref{eq2.6}) and of
the description of $\Delta(\u'\cap\p')$ in terms of diagrams given by (\ref{eq2.3}).
\end{proof}

\begin{coro}\label{coro3.3}
If
$\left(\det\bullet\right)^{\lambda_{r_0}+\frac{m_{r_0}-\alpha(\chi_2)}{2}}\otimes\theta_\chi(A_{\q'}(\lambda'))$
is non zero then $\lambda+2\rho(\u\cap\p)$ is one of these $K$-types.
\end{coro}
It remains to calculate the infinitesimal character of this $(\g,K)$-module.
\subsection{Infinitesimal character}
Infinitesimal characters of theta lifts have been determined by Przebinda.
\begin{theo}[\cite{Pr}]\label{theo3.5} If $\pi'$ has infinitesimal character 
$$\frac{\alpha(\chi_1)}{2}+\left(t_1,\dots,t_{a'+b'}\right),$$
then $\theta(\pi')$ has infinitesimal character:
$$\frac{\alpha(\chi_2)}{2}+\left(t_1,\dots,t_{a'+b'},\frac{n_{r_0}-2i+1}{2}\right)\ 1\leq i\leq n_{r_0}.$$
\end{theo}

The infinitesimal character of
$A_{\q'}(\lambda')$ is:
$$\left(\lambda_i'+\frac{m_i'}{2}+\frac{n_i+1-2k}{2}\right)\ i\neq r_0\ 1\leq k\leq n_i$$
(see lemma \ref{lemm2.1}).

So by theorem \ref{theo3.5} the infinitesimal character of
$\left(\det\bullet\right)^{\lambda_{r_0}+\frac{m_{r_0}-\alpha(\chi_2)}{2}}\otimes\theta(A_{\q'}(\lambda'))$
is (using the definition of $\lambda'$):
\begin{equation}
\begin{split}
\left(\lambda_i+\frac{n_{r_0}+m_i'}{2}+\frac{n_i-2k+1}{2}\ i<r_0\ 1\leq k\leq
n_i\right.\\
,\lambda_{r_0}+\frac{m_{r_0}}{2}+\frac{n_{r_0}-2k+1}{2}\ 1\leq k\leq
n_{r_0}\\
\left.\lambda_i+\frac{-n_{r_0}+m_i'}{2}+\frac{n_i-2k+1}{2}\ i>r_0\ 1\leq
  k\leq n_i\right).
\end{split}
\end{equation}
But by the formula for the $(m'_i) $ (see equation (\ref{eq3.4})) this is equal to the infinitesimal character of
$A_\q(\lambda)$. Finally we have:
\begin{coro}\label{coro3.4}
If
$\left(\det\bullet\right)^{\lambda_{r_0}+\frac{m_{r_0}-\alpha(\chi_2)}{2}}\otimes\theta_\chi(A_{\q'}(\lambda'))=A_\q(\lambda)$
is non zero its infinitesimal character is $\lambda+\rho_\q$.
\end{coro}

\section{Geometric applications}\label{sect4}

\subsection{Matsushima formula}
Let $X_{a,b}$ be the symmetric space associated to $U(a,b)$, this is the
quotient
$$G/K=U(a,b)/(U(a)\times U(b)).$$
Let $\Gamma$ be a discrete cocompact subgroup of $G=U(a,b)$ acting freely on
$X_{a,b}$. Let $C(\pi)$ be the constant by which the Casimir element acts on an ireductible unitary
representation $\pi$ of $G$. According to Hodge theory and Kuga's lemma, we have
that:
\begin{eqnarray}\label{eq4.1}
H^*(\Gamma\bk X_{a,b})=\oplus_{\left\{\substack{\pi\\C(\pi)=0}\right.} m(\pi,\Gamma)\hom_K(\bigwedge^*\mathfrak{p},\pi)
\end{eqnarray}
where $m(\pi,\Gamma)$ is the multiplicity of $\pi$ in $L^2(\Gamma\bk
G)$.

\begin{defi}\label{defi4.1}
 An irreducible unitarisable $(\g,K)$-module $\pi$ such that:
\begin{itemize}
\item $C(\pi)=0$ and
\item $\hom_K(\bigwedge^*\mathfrak{p},\pi)\neq 0$
\end{itemize}
is called cohomological. 
\end{defi}
Clearly the representation $A_\q(0)$ is cohomological for every $\q$ (because it has
the infinitesimal character of the trivial representation and it
contains $V(\q)$). Indeed by a theorem of Vogan and Zuckerman :

\begin{theo}[\cite{VZ}]\label{theo4.1}
Every cohomological representation is of the form $A_\q(0)$ and:
\begin{equation}\label{prim}
\hom_K(\bigwedge^{i-R(\q)}\mathfrak{p},A_\mathfrak{q}(0))=\hom_{K\cap L}(\bigwedge^i\p\cap\mathfrak{l},V(\mathfrak{q})).
\end{equation}
\end{theo}

\begin{rema}\label{rema4.1}
By the unicity in theorem \ref{theo2.0} and theorem \ref{prop2.1}, we have $A_\q(0)\simeq A_{\q'}(0)$ if and
only if the theta stable algebra have the same associated pair of
partitions.
\end{rema}

By this theorem part of the structure of the cohomology of $H^*(\Gamma\bk X_{a,b})$ is
contained in the numbers $m(A_\q,\Gamma)$. Remind that the following
theorem was stated in the:

\begin{theo}[\cite{Li1}]\label{theo4.2}
Let $(\q,L)$ be a theta stable algebra of $U(a,b)$ such that
$$L=U(x,y)\times \text{a compact group},$$
and suppose that $2(x+y)>a+b$ then there exists a discrete cocompact
subgroup of $U(a,b)$ such that:
$$m(A_\q,\Gamma)\neq 0.$$
\end{theo}
We want to generalize this theorem to other parabolic theta stable algebras
$\q$. As we explained in the introduction, the main tool of the proof is the use of global
theta correspondance and the first point of corollary \ref{coro3.1}. We
recall some facts on global theta correspondance.

\subsection{Global theta correspondance}
Let $F$ be a totally real numberfield and $E$ a totally complex quadratic
extension of $F$. Let $\epsilon_{E/F}$ be the
 character of the idele of $F$ associated to the quadratic extension
 $E/F$. Let $V_1$ (resp. $V_2$) be a
$E$-vector space with a $\eta$ (resp. $-\eta$) anisotropic hermitian form $(,)_1$
(resp. $(,)_2$). Let $n_i=\dim V_i$. We also let $G_i=U(V_i)$. Let
$\chi=(\chi_1,\chi_2)$ be a pair of characters of the idele of $E$ such that:
$${\chi_1}_{|\A_F^\times}=\epsilon_{E/F}^{n_{2}}\text{ and }{\chi_2}_{|\A_F^\times}=\epsilon_{E/F}^{n_{1}}$$
for $i=1,2$.

\subsection{Definition}
The metaplectic representation of $G_1(F\otimes\R)\times G_2(F\otimes\R)$ can be globalized 
to a representation $S_{\chi}$ of $G_1(\A)\times G_2(\A)$. Let
$\pi_1=\otimes_v\pi_{1,v}$ be an automorphic representation of $G_1(F)\bk G_1(\A)$. The theta lift of $\pi_1$
is the application:
$$\Theta:\pi_1^\vee\otimes S_\chi\rightarrow \mathcal{A}(G_2(F)\bk G_2(\A))$$
which associates to $f\otimes \vhi$ the automorphic function:
\begin{equation}\label{eq4.2}
\theta_{\psi,\chi}(f,\vhi)(h)=\int_{G_1(F)\bk G_1(\A)}f(g)\Theta_{\chi}(g,h,\vhi)dg
\end{equation}
where $\Theta_{\chi}(g,h,\vhi)$ is the theta kernel. Since we have assumed
that both $V_1$ and $V_2$ are anisotropic, there is no issue of convergence
(the quotients $G_i(F)\bk G_i(\A)$ are compact). Let
$\Theta_\chi(\pi_1,V_2)$ be the image of the application theta. 
According to S.S. Kudla and S. Rallis:
\begin{theo}[\cite{KR0} prop 7.1.2]\label{theo4.3} If $\Theta_\chi(\pi_1,V_2)\neq 0$ then it is a finite
  sum of irreducible automorphic representations. Furthermore if
  $\pi_2=\otimes_v\pi_{2,v}$ is an irreducible componant of
  $\theta_\chi(\pi_1,V_2)$ then for all archimedean
  places we have that:
$$\pi_{2,v}=\theta_\chi(\pi_{1,v}).$$
\end{theo}
Using the theory of $L$-functions and the Rallis inner product formula, one
can prove the following result:
\begin{theo}\label{theo4.4} Let $\pi_1=\otimes_v\pi_{1,v}$ be an automorphic
  representation on $G_1$ and suppose
\begin{itemize}
\item  on the one hand that $\dim V_2\geq 2\dim V_1+1$ 
\item and on the other hand that for every archimedean places $v$
$$\theta_\chi(\pi_{1,v},V_{2,v})\neq 0$$
then
$$\Theta_\chi(\pi_{1},V_{2})\neq 0.$$
\end{itemize}
\end{theo}

\begin{defi}\label{defi4.2}
Let $a,b\in \Z_{\geq 0}$. A parabolic theta stable algebra of $U(a,b)$ is
said to be convergent if there exists
\begin{itemize}
\item two sequences in $\Z_{\geq 0}$, $a_0,\dots,a_n$ and $b_0,\dots,b_n$ such that
\begin{eqnarray}
a_n=a \text{ and } b_n=b\\
 a_{i+1}+b_{i+1}&>& 2(a_i+b_i)\text{ if }0\leq i\leq n-1,\\ 
a_{i-1}+b_{i-1}&\leq& \min( a_i,b_i)\text{ if }2\leq i\leq n-1.
\end{eqnarray}
\item  for all $1\leq i\leq n$ a standard parabolic theta
  stable algebra $\q_i$ of $U(a_i,b_i)$ such that
\begin{itemize}
\item the Levi $L_0$ of $\q_0$ is compact,
\item $\q_n=\q$ and
\item for all $0\leq i\leq n-1$, $\theta(\q_{i-1})=\q_i$ (see definition \ref{defi3.2.1}).
\end{itemize}
\end{itemize}
\end{defi}

\begin{theo}\label{theo4.5}
Let $\q$ be a convergent theta stable algebra of $U(a,b)$ then 
 there exists a discrete cocompact
subgroup of $U(a,b)$ such that:
$$m(A_\q,\Gamma)\neq 0.$$
\end{theo}
\begin{proof}
We can use the notation of definition \ref{defi4.2} since $\q$ is assumed to be
convergent. There exists a  parameter $\lambda_i\in\C^{a_i}\times\C^{b_i}$
verifying (\ref{eq2.5}) and such that
$$\theta(A_{\q_{i-1}}(\lambda_{i-1}))=A_{\q_i}.$$
Since $L_0$ is compact, $A_{\q_0}$ is a discrete series. Using the main theorem of De George and
Wallach article \cite{DW1}, we can choose $\Gamma_0$ such that
$m(A_{\q_0},\Gamma_0)\neq 0$. Then using recursively global theta correspondance
(i.e theorem \ref{theo4.4}), we find a sequence of discrete subgroups
$(\Gamma_i)_{1\leq i\leq n}$ such that:
$$m(A_{\q_i}(\lambda_i),\Gamma_i)\neq 0.$$
The case $i=n$ proves the theorem.
\end{proof}

One can of course obtain more precise statements (about the possible
choice of $\Gamma$ for example). Furthermore using the
method of \cite{C1}, one can even obtain some results on the growth of the
multiplicity of $m(A_\q(0),\Gamma)$ in the tower of arithmetic congruence groups,
see \cite{C2} for more details.

\end{document}